\documentclass[11pt]{article}
\usepackage{amsfonts}
\usepackage{color}
\newtheorem{theo}{Theorem}[section]
\newtheorem{lem}[theo]{Lemma}

\newtheorem{cor}[theo]{Corollary}

\newcommand{\mysection}[1]{\section{#1} \setcounter{equation}{0}}
\newcommand{\proof}{{\sc Proof.} \quad}
\newcommand{\proofc}{{\sc Proof} \ }
\newcommand{\be}{\begin{equation} \label}
\newcommand{\ee}{\end{equation}}
\newcommand{\bea}{\begin{eqnarray}\label}
\newcommand{\eea}{\end{eqnarray}}
\newcommand{\bas}{\begin{eqnarray*}}
\newcommand{\eas}{\end{eqnarray*}}
\newcommand{\bit}{\begin{itemize}}
\newcommand{\eit}{\end{itemize}}
\newcommand{\qed}{\hfill$\Box$ \vskip.2cm}
\newcommand{\nn}{\nonumber}
\newcommand{\R}{\mathbb{R}}

\newcommand{\eps}{\varepsilon}

\newcommand{\io}{\int_\Omega}
\newcommand{\bom}{\overline{\Omega}}
\newcommand{\mint}{- \hspace*{-3.5mm} \int}
\newcommand{\hs}{\hspace*}

\newcommand{\ov}{\overline}
\newcommand{\un}{\underline}
\newcommand{\abs}{\\[5pt]}
\newcommand{\tm}{T_{max}}

\newcommand{\uz}{\underline{z}}
\newcommand{\oz}{\overline{z}}
\newcommand{\xis}{\xi_\star}
\newcommand{\ms}{m_\star}
\newcommand{\omc}{\ov{m}_\star}
\newcommand{\umc}{\un{m}_\star}
\newcommand{\set}{{\cal S}}
\begin{document}
\enlargethispage{10mm}
\title{Identifying secondary critical mass levels\\
in two- and higher-dimensional Keller-Segel systems}
\author{
{ Yulan Wang \footnote{wangyulan-math@163.com}}\\
{\small School of Science, Xihua University,}\\
{\small 610039 Chengdu, China}
\and
Michael Winkler\footnote{michael.winkler@math.uni-paderborn.de}\\
{\small Universit\"at Paderborn, Institut f\"ur Mathematik}\\
{\small 33098 Paderborn, Germany}
}
\date{}
\maketitle
\begin{abstract}
\noindent
  The parabolic-elliptic Keller-Segel system
  \bas
	\left\{ \begin{array}{l}
	u_t = \Delta u - \nabla \cdot (u\nabla v), \\[1mm]
	0 = \Delta v - \mu + u,
	\qquad \mu:=\frac{1}{|\Omega|} \io u,
 	\end{array} \right.
	\qquad \qquad (\star)
  \eas
  is considered under homogeneous Neumann boundary conditions in the ball $\Omega=B_R(0)\subset \R^n$, where $n\ge 2$
  and $R>0$. \abs
  The main results identify a mass level $\ms(n,R)>0$ as critical in the sense that
  \bit
  \item
  whenever $0<m<\ms(n,R)$, the constant equilibrium of ($\star$) given by $u\equiv v\equiv \frac{m}{|B_R|}$
  is locally asymptotically stable with respect to the norm in $L^\infty(B_R)$,
  \eit
  but that
  \bit
  \item
  if $m>\ms(n,R)$, then any solution of ($\star$) emanating from
  nonnegative and nonconstant continuous radial initial data $u_0$ with $\int_{B_R} u_0=m$ and
  $\frac{1}{|B_r|} \int_{B_r} u_0 \ge \frac{m}{|B_R|}$ for all $r\in (0,R)$
  blows up in finite time.
  \eit
  The value of $\ms(n,R)$ can be determined in an essentially explicit manner in dependence on principal zeroes of
  appropriate Bessel functions of the first kind; in particular, $\ms(2,R)\approx 14.68197 \cdot \pi$ for all $R>0$.
  Accordingly, as a supplement to known mass-dependent blow-up dichotomies these findings reveal a genuinely secondary
  critical mass phenomenon for ($\star$) especially also in its two-dimensional version in which no precedent of this
  form seems known.\abs
\end{abstract}
\newpage
\section{Introduction}\label{intro}
Critical mass phenomena belong to the most striking subtleties of Keller-Segel type cross-diffusion systems
(\cite{perthame_book}).
In contexts of initial-boundary value problems of the form
\be{0}
    	\left\{ \begin{array}{rcll}
	u_t &=& \Delta u - \nabla \cdot (u\nabla v),
	\qquad & x\in B_R, \ t>0, \\[1mm]
	0 &=& \Delta v - \mu + u,
	\qquad \mu= \mint_{B_R} u_0,
	\qquad & x\in B_R, \ t>0, \\[1mm]
	& & \hspace*{-13mm}
	\frac{\partial u}{\partial\nu}=\frac{\partial v}{\partial\nu}=0,
	\qquad & x\in\partial B_R, \ t>0, \\[1mm]
	& & \hspace*{-13mm}
	u(x,0)=u_0(x),
	\qquad & x\in\ B_R,
 	\end{array} \right.
\ee
the most classical findings in this direction assert the case when the spatial domain
$B_R:=B_R(0)$ herein is planar, for any suitably regular nonnegative and radially symmetric $u_0$ fulfilling $\int_{B_R} u_0<
m_c(2,R):=8\pi$
a globally defined bounded classical solution exists, whereas for any $m>8\pi$ one can find radial initial data
such that $\int_{B_R} u_0=m$ but that a corresponding temporally local classical solution blows up in finite time
(\cite{nagai1995}, \cite{biler1998}, \cite{senba_suzuki}).
Extensions to various two-dimensional relatives have confirmed a considerable stability of this phenomenon with respect to
modifications in the second sub-problem of (\ref{0}), including passage to slightly more complex elliptic or even
parabolic problems (\cite{nagai2001}, \cite{suzuki_JMPA}, \cite{blanchet_dolbeault_perthame_EJDE2006},
\cite{kozono_sugiyama_JEE2008}, \cite{biler_cieslak_karch_zienkiewicz_DCDS2017}, \cite{herrero_velazquez},
\cite{horstmann_wang}).
{Some further detections on critical mass phenomena in some closely related chemotaxis systems can be found recorded in
\cite{mao_liu_yuxiang}, \cite{mao_yuxiang}, \cite{yuxiang}, \cite{xiang_yang}, \cite{yang_lan}, \cite{ding_win}, \cite{flw} and
\cite{taowin_JEMS}, for instance.
}
\abs
In higher-dimensional versions of (\ref{0}), at arbitrarily small mass levels some solutions blow up,
meaning that maintaining the focus on the mere occurrence of explosions
would lead to the conclusion that a correspndingly defined threshold should trivially satisfy
$m_c(n,R)=0$ for each $n\ge 3$ and $R>0$ (\cite{nagai1995}, \cite[Lemma 3.3]{win_MATANN}).\abs
More recently,
a refined analysis of the destabilizing action of cross-diffusion in (\ref{0})
has identified a pivotal influence of the total mass $\int_{B_R} u_0$ on the dynamics of (\ref{0})
not only in the planar but also in higher-dimensional settings.
The study in \cite{win_MATANN}, namely, examines how far singularity formation may be enforced by
initially present perturbations of homogeneous steady states, with a focus set on concentration-increasing perturbations.
More precisely, it has been found in \cite[Theorems 1.1 and 1.3]{win_MATANN} that whenever $n\ge 2$ and $R>0$, letting
\bea{set}
	\set(m,R) &:=& \bigg\{ u_0\in C^0(\ov{B}_R) \ \bigg| \
	\mbox{$u_0$ is nonnegative and radial with } \int_{B_R} u_0=m \mbox{\quad and} \nn\\
	& & \hspace*{28mm}
	\frac{1}{|B_r|} \int_{B_r} u_0 \ge \frac{m}{|B_R|} \mbox{ for all $r\in (0,R)$,}
	\quad but \quad
	u_0\not\equiv \frac{m}{|B_R|}
	\bigg\},
\eea
denote the set of all radial initial data at mass level $m$ which represent mass concentration stronger than that of the
constant equilibrium $u\equiv \frac{m}{|B_R|}$, we have that
both numbers
\be{omc}
	\omc(n.R):=\inf \Big\{ m>0 \ \Big| \
	\mbox{For all $u_0\in \set(m,R)$, (\ref{0}) admits a solution blowing up in finite time}
	\Big\}
\ee
and
\be{umc}
	\umc(n.R):=\sup \Big\{ m>0 \ \Big| \
	\mbox{(\ref{0}) possesses a global bounded solution for some $u_0\in \set(m,R)$} \Big\}
\ee
are well-defined and positive; here, as usual, a global bounded solution of (\ref{0}) is to be understood as
a classical solution $(u,v)$
on $B_R\times (0,\infty)$ for which $u$ belongs to $L^\infty(B_R\times (0,\infty))$, while by a solution blowing up in finite time
we mean a solution $(u,v)$ which is defined in $B_R\times (0,\tm)$ for some $\tm\in (0,\infty)$ and satisfies
$u\not\in L^\infty(B_R\times (0,\tm))$.
In particular, these properties of $\umc$ and $\omc$ imply that at each accordingly supercritical mass level $m>\omc(n,R)$,
the corresponding homogeneous equilibrium $\frac{m}{|B_R|}$ of (\ref{0}) is highly unstable in the sense that
any concentration-increasing perturbation thereof will lead to finite-time blow-up,
while throughout the subcritical range $(0,\umc(n,R))$ of mass values, no such drastic instability feature is present.\abs
According to its principal design mainly based on the use of comparison principles, however, the approach in \cite{win_MATANN}
seems uncapable of providing information on the precise location of these critical numbers.
In fact, beyond the identity $\umc(2,\cdot)\equiv 8\pi$,
as resulting from \cite{nagai1995} and \cite{suzuki_book}, and the rough one-sided
inequality $\umc(n,R)\ge 2n|B_1| R^{n-2}$, as implied for $n\ge 3$ and $R>0$ by \cite[Theorem 1.3]{win_MATANN},
farther-reaching knowledge appears to be widely lacking.
In particular, it even seems to have remained unclear so far whether $\umc(n,R)$ and $\omc(n,R)$
can differ for some $n\ge 2$ and $R>0$, or must coincide; especially in the two-dimensional case, to date it consequently
appears yet open whether the corresponding numbers $\omc(2,R)$ indeed play the role of a genuinely secondary critical mass
for some $R>0$, or whether they all actually coincide with the threshold $8\pi$ that has formed the dominating core
in the history of discoveries on mass criticality in the planar version of (\ref{0})
(\cite{nagai1995}, \cite{biler1998}, \cite{suzuki_book}).\abs
{\bf Main results.} \quad
The present study now intends to develop methods by which, inter alia, the sizes of
both $\omc(n,R)$ and $\umc(n,R)$ can precisely be determined.
In fact, in Section \ref{sect3} we shall see that a comparison argument involving appropriately rescaled Bessel functions
$\Theta$ (Lemma \ref{lem1})
yields a statement on local asymptotic stability of constant steady states, and hence especially on exclusion of blow-up
near such equilibria, under a condition on the considered mass level
which explicitly depends on the location of the associated principal zero of $\Theta$ (Lemma \ref{lem3}).
Section \ref{sect4} will thereafter indicate that these choices of comparison functions indeed capture essential features
of the actual solution behavior: An apparently novel type of refinement in a moment-based contradition argument,
namely, will reveal that under an assumption on the mass size essentially complementary to the above,
a certain moment-type functional, involving $\Theta$ as a weight, must cease to exist in finite time whenever $u_0$ lies
in the corresponding set of concentration-increasing perturbations of homogeneity according to (\ref{set}).\abs
To precisely formulate the accordingly obtained outcome, here and below we shall abbreviate
\be{init}
	u_0\in C^0_{rad}(\ov{B}_R) :=
	\Big\{ \varphi\in C^0(\ov{B}_R) \ \Big| \ \varphi \mbox{ is radially symmetric } \Big\}
\ee
for $n\ge 2$ and $R>0$, with radial symmetry exclusively referring to the choice of the origin as the respective center.
Apart from that, given $n\ge 2$ we let
\be{J}
	J_\frac{n}{2}(\xi):=\sum_{k=0}^\infty \frac{(-1)^k}{k! \Gamma(k+\frac{n}{2}+1)} (\frac{\xi}{2})^{2k+\frac{n}{2}},
	\qquad \xi>0,
\ee
represent the Bessel function of the first kind, and we let $\xis=\xis(n)$ denote its first positive zero.\abs
Our main results can then be formulated as follows.
\begin{theo}\label{theo5}
  Let $n\ge 2$ and $R>0$, and let
  \be{ms}
	\ms(n,R):=\frac{\omega_n R^{n-2}}{n} \cdot \xis^2(n),
  \ee
  where $\omega_n:=n|B_1|$,
  and where $\xis(n)$ denotes the first positive zero of the Bessel function $J_\frac{n}{2}$ in (\ref{J}).\abs
  i) \ For any choice of $m>0$ such that
  \be{5.3}
	m<\ms(n,R),
  \ee
  one can find $\eps=\eps(m,R)>0$ and $\alpha=\alpha(m,R)>0$ with the property that whenever $u_0$ satisfies (\ref{init})
  with $\int_{B_R} u_0=m$ and
  \be{5.4}
	\Big\| u_0-\frac{m}{|B_R|} \Big\|_{L^\infty(B_R)} \le \eps,
  \ee
  the problem (\ref{0}) admits a unique global classical solution $(u,v)$ with
  $u\in C^0(\ov{B}_R \times [0,\infty)) \cap C^{2,1}(\ov{B}_R \times (0,\infty))$ and
  $v\in C^{2,0}(\ov{B}_R \times (0,\infty))$, for which there exists $C=C(u_0)>0$ such that
  \be{5.5}
	\Big\| u(\cdot,t)-\frac{m}{|B_R|} \Big\|_{L^\infty(B_R)} \le C e^{-\alpha t}
	\qquad \mbox{for all } t>0.
  \ee
  ii) \ If $m>0$ is such that 
  \be{5.1}
	m>\ms(n,R),
  \ee
  then any function $u_0$ fulfilling (\ref{init}) with $\int_{B_R} u_0=m$ and
  \be{5.2}
	\frac{1}{|B_r|} \int_{B_r} u_0 \ge \frac{m}{|B_R|}
	\quad \mbox{for all } r\in (0,R),
	\qquad \mbox{but} \qquad
	u_0\not\equiv \frac{m}{|B_R|},
  \ee
  has the property that the corresponding solution of (\ref{0}) blows up in finite time; that is, there exist $\tm\in (0,\infty)$
  as well as functions
  $u\in C^0(\ov{B}_R \times [0,\tm)) \cap C^{2,1}(\ov{B}_R \times (0,\tm))$ and
  $v\in C^{2,0}(\ov{B}_R \times (0,\tm))$ such that $(u,v)$ solves (\ref{0}) classically in $B_R\times (0,\tm)$, and that
  \bas
	\limsup_{t\nearrow\tm} \|u(\cdot,t)\|_{L^\infty(B_R)} =\infty.
  \eas
\end{theo}
As an immediate consequence, we can thereby make sure that both $\omc(n,R)$ and $\umc(n,R)$ coincide
with $\ms(n,R)$ whenever $n\ge 2$ and $R>0$:
\begin{cor}\label{cor55}
  For $n\ge 2$ and $R>0$, let $\omc(n,R)$, $\umc(n,R)$ and $\ms(n,R)$ be as in (\ref{omc}), (\ref{umc}) and (\ref{ms}).
  Then
  \bas
	\omc(n,R)=\umc(n,R)=\ms(n,R)
	\qquad \mbox{for all $n\ge 2$ and } R>0.
  \eas
\end{cor}
{\bf Remark.} \quad
i) \ Recalling that $\xis(2)\approx 3.83170$,
from the above we particularly obtain that
\bas
	\ms(2,R)\approx 14.68197 \cdot \pi,
\eas
which especially means that the number $\ms(2,R)$ is yet independent of $R$ but strictly
larger than the corresponding purely explosion-focused critical mass level $m_c(2,\cdot)\equiv 8\pi$.
Accordingly, Theorem \ref{theo5} and Corollary \ref{cor55} thus identify
a genuinely secondary critical mass phenomenon in the two-dimensional
version of (\ref{0}).
We emphasize that up to a single exceptional finding focusing on effects of boundary conditions (\cite{flw}),
we are not aware of any study in which a comparable feature has been discovered.\abs
ii) \ Let us moreover underline that if and only if $n\ge 3$,
the phenomenon detected above is domain-dependent in the sense that then
$\ms(n,\cdot)$ is strictly increasing on $(0,\infty)$.
\mysection{Local existence and transformation to a scalar problem}\label{sect2}
The following basic result on local solvability in (\ref{0}) is essentially well-known
(see, e.g., the arguments in \cite{biler1998} and \cite{djie_win}).	
\begin{lem}\label{lem_loc}
  Assume that $n\ge 2$ and $R>0$, and that (\ref{init}) holds.
  Then there exist $\tm\in (0,\infty]$ and uniquely determined functions
  \bas
	\left\{ \begin{array}{l}
	u \in C^0(\ov{B}_R\times [0,\tm)) \cap C^{2,1}(\ov{B}_R\times (0,\tm)) \qquad \mbox{and} \\[1mm]
	v\in \bigcap_{q>n} L^\infty_{loc}([0,\tm);W^{1,q}(B_R)) \cap C^{2,0}(\ov{B}_R\times (0,\tm)),
	\end{array} \right.
  \eas
  such that
  \bas
	\int_{B_R} v(\cdot,t)=0
	\qquad \mbox{for all } t\in (0,\tm),
  \eas
  that $(u(\cdot,t),v(\cdot,t))$ is radially symmetric for all $t\in (0,\tm)$, that $(u,v)$
  solves (\ref{0}) classically in $B_R\times (0,\tm)$, and that
  \be{ext}
	\mbox{if $\tm<\infty$, \quad then \quad }
	\limsup_{t\nearrow\tm} \|u(\cdot,t)\|_{L^\infty(B_R)} = \infty.
  \ee
  Moreover, $u>0$ in $\bom\times (0,\tm)$ and
  \be{mass}
	\int_{B_R} u(\cdot,t) = \int_{B_R} u_0
	\qquad \mbox{for all } t\in (0,\tm).
  \ee
\end{lem}
Given $u_0$ fulfilling (\ref{init}) and letting $\tm$ and $(u,v)$ be as accordingly
provided by Lemma \ref{lem_loc}, following a standard transformation precedure (\cite{JL})
we introduce the corresponding cumulated densities by defining
\be{w}
	w(s,t):=\int_0^{s^\frac{1}{n}} \rho^{n-1} u(\rho,t) d\rho,
	\qquad s\in [0,R^n], \ t\in [0,\tm),
\ee
and
\be{w0}
	w_0(s):=\int_0^{s^\frac{1}{n}} \rho^{n-1} u_0(\rho) d\rho,
	\qquad s\in [0,R^n],
\ee
noting that then
\be{ws}
	w_s(s,t)=\frac{1}{n} u(s^\frac{1}{n},t)
	\quad \mbox{and} \quad
	w_{ss}(s,t)=\frac{1}{n^2} s^{\frac{1}{n}-1} u_r(s^\frac{1}{n},t),
	\qquad s\in (0,R^n), \ t\in (0,\tm).
\ee
According to Lemma \ref{lem_loc}, we thus obtain that $w_s>0$ in $[0,R^n] \times [0,\tm)$, and that $w$ solves
the Dirichlet problem
\be{0w}
    	\left\{ \begin{array}{ll}
	w_t = n^2 s^{2-\frac{2}{n}} w_{ss} + nww_s - \mu s w_s,
	\qquad & s\in (0,R^n), \ t\in (0,\tm), \\[1mm]
	w(0,t)=0, \quad w(R^n,t)=\frac{m}{\omega_n},
	\qquad & t\in (0,\tm), \\[1mm]
	w(s,0)=w_0(s),
	& s\in (0,R^n),
 	\end{array} \right.
\ee
with $m:=\io u_0$ and $\mu:=\frac{nm}{\omega_n R^n}$, where $\omega_n:=n|B_1|$.\abs
As an additional preliminary ingredient, let us recall the following boundedness criterion that has been
derived on the basis of a Bernstein-type comparison argument in \cite[Lemma 4.1]{win_MATANN}.
\begin{lem}\label{lem79}
  Let $n\ge 2$ and $R>0$, and suppose that $u_0$ satisfies (\ref{init}) and is such that the function $w$ defined in
  (\ref{w}) has the property that
  \be{79.1}
	\sup_{(s,t)\in (0,R^n)\times (0,\tm)} \frac{w(s,t)}{s} < \infty.
  \ee
  Then there exists $C>0$ such that
  \be{79.2}
	w_s(s,t) \le C
	\qquad \mbox{for all $s\in (0,R^n)$ and } t\in (0,\tm).
  \ee
\end{lem}
In large parts of our analysis, instead of directly referring to the functions $w$ defined above we shall rather
consider their deviations from the corresponding steady-state solutions: Namely, letting $w$ be as in (\ref{w}) we define
defining
\be{z}
	z(s,t):=w(s,t)-\frac{\mu s}{n},
	\qquad s\in [0,R^n], \ t\in [0,\tm).
\ee
and note that then, as a direct consequence of (\ref{0w}),
\be{0z}
    	\left\{ \begin{array}{ll}
	z_t = n^2 s^{2-\frac{2}{n}} z_{ss} + nzz_s + \mu z,
	\qquad & s\in (0,R^n), \ t\in (0,\tm), \\[1mm]
	z(0,t)=z(R^n,t)=0,
	\qquad & t\in (0,\tm), \\[1mm]
	z(s,0)=w_0(s)-\frac{\mu s}{n},
	& s\in (0,R^n).
 	\end{array} \right.
\ee
\mysection{A comparison-based approach toward local asymptotic stability}\label{sect3}
Our strategy toward the derivation of Theorem \ref{theo5} i) will consist in the use of appropriately decaying
sub- and supersolutions to be used as comparison functions in an analysis of (\ref{0z}).
Our selections in this regard (cf.~(\ref{oz}) and (\ref{uz}) will at their core be determined by suitably rescaled
versions of the Bessel functions from (\ref{J}), to be defined and described in the following.
\begin{lem}\label{lem1}
  Given $n\ge 2$ and $R>0$, we let
  \be{1.1}
	\Theta(s):=s^{-\frac{1}{2}} J_\frac{n}{2} \bigg( \frac{\xis s^\frac{1}{n}}{R}\bigg),
	\qquad s>0,
  \ee
  where $\xis=\xis(n)$ is as specified in Theorem \ref{theo5}.
  Then $\Theta$ belongs to $C^2((0,\infty))$ and satisfies
  \be{1.2}
	s\Theta_{ss}(s) + 2\Theta_s(s) = - \frac{\xis^2}{n^2 R^2} s^{\frac{2}{n}-1} \Theta(s)
	\qquad \mbox{for all } s>0.
  \ee
  Moreover,
  \be{1.3}
	\Theta(s) \to \frac{\xis^\frac{n}{2}}{(2R)^\frac{n}{2} \Gamma(\frac{n+2}{2})}
	\qquad \mbox{as } s\searrow 0
  \ee
  and
  \be{1.4}
	\Theta_s(s) < 0
	\qquad \mbox{for all } s\in (0,R^n],
  \ee
  and we have
  \be{1.5}
	s\Theta_s(s) \to 0
	\qquad \mbox{as } s\searrow 0.
  \ee
\end{lem}
\proof
  Abbreviating $\kappa:=\frac{\xis}{R}$ and writing $\xi=\kappa s^\frac{1}{n}$ for $s>0$, we use (\ref{1.1}) to compute
  \bas
	\Theta_s(s)=\frac{\kappa}{n} s^{\frac{1}{n}-\frac{3}{2}} J_\frac{n}{2}'(\xi) - \frac{1}{2} s^{-\frac{3}{2}} J_\frac{n}{2}(\xi),
	\qquad s>0,
  \eas
  and
  \bas
	\Theta_{ss}(s)
	= \frac{\kappa^2}{n^2} s^{\frac{2}{n}-\frac{5}{2}} J_\frac{n}{2}''(\xi)
	+ \frac{\kappa}{n}\Big( \frac{1}{n}-2\Big) s^{\frac{1}{n}-\frac{5}{2}} J_\frac{n}{2}'(\xi)
	+ \frac{3}{4} s^{-\frac{5}{2}} J_\frac{n}{2}(\xi),
	\qquad s>0.
  \eas
  Due to the well-known fact that
  $\xi^2 J_\frac{n}{2}''(\xi) + \xi J_\frac{n}{2}'(\xi) + (\xi^2-\frac{n^2}{4}) J_\frac{n}{2}(\xi)=0$ for all $\xi>0$,
  this implies that
  \bas
	s\Theta_{ss}(s) + 2\Theta_s(s)
	&=& \frac{\kappa^2}{n^2} s^{\frac{2}{n}-\frac{3}{2}} J_\frac{n}{2}''(\xi)
	+ \bigg\{ \frac{\kappa}{n}\Big(\frac{1}{n}-2\Big) + 2\cdot\frac{\kappa}{n}\bigg\} \cdot s^{\frac{1}{n}-\frac{3}{2}}
		J_\frac{n}{2}'(\xi)
	+ \Big\{ \frac{3}{4} - 2\cdot\frac{1}{2}\Big\} \cdot s^{-\frac{3}{2}} J_\frac{n}{2}(\xi) \\
	&=& \frac{1}{n^2} \cdot s^{-\frac{3}{2}} \cdot \Big\{
	\xi^2 J_\frac{n}{2}''(\xi) + \xi J_\frac{n}{2}'(\xi) - \frac{n^2}{4} J_\frac{n}{2}(\xi)\Big\} \\
	&=& \frac{1}{n^2} \cdot s^{-\frac{3}{2}} \cdot \Big\{ -\xi^2 J_\frac{n}{2}(\xi)\Big\},
	\qquad s>0, \ \xi=\kappa s^\frac{1}{n},
  \eas
  from which (\ref{1.2}) immediately follows by definition of $\kappa$.\abs
  To verify (\ref{1.3}), we only need to note that due to the series representation of $J_\frac{n}{2}$,
  \be{1.6}
	\Theta(s)= \sum_{k=0}^\infty a_k s^\frac{2k}{n}
	\qquad \mbox{for all } s>0,
  \ee
  where $a_k:=\frac{(-1)^k}{k! \Gamma(k+\frac{n}{2}+1)} \cdot (\frac{\kappa}{2})^{2k+\frac{n}{2}}$ for $k\ge 0$ and hence
  $a_0=\frac{1}{\Gamma(\frac{n}{2}+1)} \cdot (\frac{\kappa}{2})^\frac{n}{2}$.\abs
  Likewise, (\ref{1.5}) follows from the fact that, as a consequence of (\ref{1.6}),
  \bas
	s^{-\frac{2}{n}} \cdot \big( s\Theta_{s}(s)\big)
	= \sum_{k=1}^\infty \frac{2k}{n} a_k s^\frac{2(k-1)}{n}
	\to \frac{2a_1}{n}
	\qquad \mbox{as } s\searrow 0,
  \eas
  while (\ref{1.4}) can be concluded from the observation that in view of (\ref{1.2}), $\Theta$ cannot attain a local minimum
  at any $s_0>0$ with $\Theta(s_0)>0$.
\qed
Using these functions as spatial profiles of comparison functions with separated structure, we can thereby
assert asymptotic stability of homogeneous equilibria at subcritical mass levels in the following style,
yet acting at the stage of the transformed variables $w$ and $z$:
\begin{lem}\label{lem2}
  Let $n\ge 2$ and $R>0$, and with $\Theta$ and $\xis$ taken from Lemma \ref{lem1}, suppose that $m>0$ is such that (\ref{5.3})
  holds.
  Then there exist $\eps_0=\eps_0(m,R)>0$ and $\alpha_0=\alpha_0(m,R)>0$ with the property that whenever $u_0$ is such that
  (\ref{init}) with $\int_{B_R} u_0=m$, and that the function $w_0$ in (\ref{w0}) satisfies
  \be{2.2}
	\Big| w_0(s) - \frac{\mu s}{n}\Big| \le \eps_0 s \Theta(s)
	\qquad \mbox{for all } s\in (0,R^n),
  \ee
  we have
  \be{2.3}
	\Big| w(s,t)-\frac{\mu s}{n}\Big| \le \eps_0 s \Theta(s) e^{-\alpha_0 t}
	\qquad \mbox{for all $s\in (0,R^n)$ and } t\in (0,\tm),
  \ee
  with $\mu=\frac{nm}{\omega_n R^n}$.
\end{lem}
\proof
  Given $R>0$ and $m>0$ fulfilling (\ref{5.3}), by definition of $\mu$ we have
  \bas
	\mu < \frac{n}{\omega_n R^n} \cdot \frac{\omega_n R^{n-2}}{n} \cdot \xis^2
	= \frac{\xis^2}{R^2},
  \eas
  so that it is possible to fix $\alpha_0=\alpha_0(m,R)>0$ small enough such that
  \be{2.4}
	\alpha_0 < \frac{\xis^2}{R^2}-\mu.
  \ee
  We next combine (\ref{1.3}) with (\ref{1.4}) to see that if we let
  $c_1=c_1(R):=\frac{\xis^\frac{n}{2}}{(2R)^\frac{n}{2} \Gamma(\frac{n+2}{2})}$, then
  \be{2.5}
	\Theta(s) \le c_1
	\qquad \mbox{for all } s\in (0,R^n),
  \ee
  while using (\ref{1.5}) we can pick $c_2=c_2(R)>0$ such that
  \be{2.6}
	s\Theta_s(s) \ge -c_2
	\qquad \mbox{for all } s\in (0,R^n).
  \ee
  Relying on (\ref{2.4}), we can thereupon choose $\eps_0=\eps_0(m,R)>0$ suitably small such that
  \be{2.8}
	n\eps_0\cdot\max \{c_1 \, , \, c_2\} \le \frac{\xis^2}{R^2}-\alpha_0-\mu.
  \ee
  We henceforth suppose that (\ref{init}) is valid with $\int_{B_R} u_0=m$, and that the corespondingly defined function $w_0$
  from (\ref{w0}) satisfies (\ref{2.2}). To see that then (\ref{2.3}) must hold, we let
  \be{oz}
	\oz(s,t):=\left\{ \begin{array}{ll}
	y(t) s\Theta(s),
	\qquad & s\in (0,R^n], \ t\ge 0, \\[1mm]
	0,
	\qquad & s=0, \ t\ge 0,
	\end{array} \right.
  \ee
  as well as
  \be{uz}
	\uz(s,t) := -\oz(s,t),
	\qquad s\in [0,R^n], \ t\ge 0,
  \ee
  where
  \be{2.9}
	y(t):=\eps_0 e^{-\alpha_0 t},
	\qquad t\ge 0.
  \ee
  Then according to Lemma \ref{lem1}, and particularly due to (\ref{1.3}), both $\oz$ and $\uz$ belong to
  $C^0([0,R^n]\times [0,\infty)) \cap C^{2,1}((0,R^n)\times [0,\infty))$ and satisfy
  \be{2.10}
	\oz(s,0)
	= \eps_0 s\Theta(s)
	\ge w_0(s) -\frac{\mu s}{n}
	\ge -\eps_0 s\Theta(s)
	= \uz(s,0)
	\qquad \mbox{for all } s\in (0,R^n)
  \ee
  as well as
  \bea{2.11}
	& & \hs{-20mm}
	\oz(0,t) = \uz(0,t) = z(0,t)=0
	\quad \mbox{and} \quad
	\oz(R^n,t)=\uz(R^n,t)=z(R^n,t)=w(R^n,t)-\frac{\mu R^n}{n}=0 \nn\\[2mm]
	& & \hs{90mm}
	\qquad \mbox{for all } t\in (0,\tm).
  \eea
  Apart from that, on the basis of (\ref{1.2}) and the identities $y'=-\alpha_0 y$,
  $(s\Theta)_s = s\Theta_s + \Theta$ and $(s\Theta)_{ss}=s\Theta_{ss} + 2\Theta_s$ we compute
  \bas
	& & \hs{-15mm}
	\oz_t- n^2 s^{2-\frac{2}{n}} \oz_{ss} - n\oz\oz_s - \mu \oz \\
	&=& y'(t) s\Theta(s)
	- n^2 s^{2-\frac{2}{n}} \cdot y(t) \cdot \Big\{ s\Theta_{ss}(s)+2\Theta_s(s)\Big\} \\
	& & - n\cdot y(t) s\Theta(s) \cdot y(t) \cdot \Big\{ s\Theta_s(s)+\Theta(s)\Big\}
	-\mu \cdot y(t) s\Theta(s) \\
	&=& - \alpha_0 y(t) s\Theta(s)
	+ \frac{\xis^2}{R^2} y(t) s\Theta(s)
	- ny^2(t) s\Theta(s) \cdot \Big\{ s\Theta_s(s)+\Theta(s)\Big\}
	- \mu y(t) s\Theta(s) \\
	&=& y(t)s\Theta(s) \cdot
	\Big\{ \frac{\xis^2}{R^2} - \alpha_0-\mu -ny(t) s\Theta_s(s) -ny(t)\Theta(s)\Big\}
	\qquad \mbox{for all $s\in (0,R^n)$ and } t>0.
  \eas
  Since $\Theta_s(s)\le 0$ for all $s\in (0,R^n)$ by (\ref{1.4}), in view of (\ref{2.5}) and (\ref{2.9}) this shows that
  \bea{2.12}
	\oz_t- n^2 s^{2-\frac{2}{n}} \oz_{ss} - n\oz\oz_s - \mu \oz
	&\ge& y(t) s\Theta(s) \cdot
	\Big\{ \frac{\xis^2}{R^2} - \alpha_0-\mu -n\eps_0 c_1\Big\} \nn\\[2mm]
	&\ge& 0
	\qquad \mbox{for all $s\in (0,R^n)$ and } t>0,
  \eea
  as a consequence of the first among the restrictions on $\eps_0$ contained in (\ref{2.8}).\abs
  In quite a similar manner, it can be seen that the function $\uz$ satisfies
  \bas
	& & \hs{-15mm}
	\uz_t- n^2 s^{2-\frac{2}{n}} \uz_{ss} - n\uz\uz_s - \mu \uz \\
	&=& - y(t) s\Theta(s) \cdot
	\Big\{ \frac{\xis^2}{R^2} - \alpha_0-\mu + ny(t)s\Theta_s(s) + ny(t)\Theta(s)\Big\}
	\qquad \mbox{for all $s\in (0,R^n)$ and } t>0,
  \eas
  whence using (\ref{2.6}) along with the nonnegativity of $\Theta$, we may draw on the seond requirement in (\ref{2.8}) to find that
  \bea{2.13}
	\uz_t- n^2 s^{2-\frac{2}{n}} \uz_{ss} - n\uz\uz_s - \mu \uz
	&\le& - y(t) s\Theta(s) \cdot
	\Big\{ \frac{\xis^2}{R^2} - \alpha_0-\mu -n\eps_0 c_2\Big\} \nn\\[2mm]
	&\le& 0
	\qquad \mbox{for all $s\in (0,R^n)$ and } t>0.
  \eea
  Now combining (\ref{2.10}) and (\ref{2.11}) with (\ref{2.12}) and (\ref{2.13}), by means of a comparison principle covering
  the present case of degenerate diffusion (cf., e.g., \cite[Lemma 2.3]{win_ct_growup_algebraic})
  we conclude that thanks to (\ref{0z})
  we have
  \bas
	\uz(s,t)\le z(s,t)\le \oz(s,t)
	\qquad \mbox{for all $s\in (0,R^n)$ and } t\in (0,\tm),
  \eas
  which in light of the definitions of $\uz, \oz$ and $z$ precisely yields (\ref{2.3}).
\qed
An application of Lemma \ref{lem79} enables us to convert this into a stability feature
of the derivative $w_s$, and hence of the original variable $u$.
\begin{lem}\label{lem3}
  Let $n\ge 2$, $R>0$ and $m>0$ be such that (\ref{5.3}) holds.
  Then there exist $\eps=\eps(m,R)>0$ and $\alpha=\alpha(m,R)>0$ such that if $u_0$ satisfies (\ref{init}) with $\int_{B_R} u_0=m$
  and
  \be{3.1}
	\|u_0-\mu\|_{L^\infty(B_R)}\le \eps
  \ee
  with $\mu=\frac{nm}{\omega_n R^n}$,
  then $\tm=\infty$, and one can find $C=C(u_0)>0$ such that
  \be{3.2}
	\|u(\cdot,t)-\mu\|_{L^\infty(B_R)} \le C e^{-\alpha t}
	\qquad \mbox{for all } t>0.
  \ee
\end{lem}
\proof
  We let $\eps_=\eps_0(m,R)>0$ and $\alpha_0=\alpha_0(m,R)>0$ be as provided by Lemma \ref{lem2}, and rely on (\ref{1.4}) to fix
  $c_1=c_1(R)>0$ fulfilling
  \be{3.3}
	\Theta(s)\ge c_1\cdot (R^n-s)
	\qquad \mbox{for all } s\in (0,R^n).
  \ee
  We furthermore pick any $q\in (1,\frac{n}{n-1})$ and set $\gamma:=\frac{q}{2q-1}$ as well as
  \be{3.4}
	\eps\equiv \eps(m,R):=\frac{nc_1 R^n \eps_0}{2}
	\qquad \mbox{and} \qquad
	\alpha\equiv \alpha(m,R):=\alpha_0\cdot (1-\gamma),
  \ee
  noting that indeed $\alpha$ is positive, since $q>1$ implies that $\gamma<1$.\abs
  Now assuming $u_0$ to comply with (\ref{init}) and to satisfy $\int_{B_R} u_0=m$ as well as (\ref{3.1}),
  we first observe that then the corrsepondingly defined functio n $w_0$ in (\ref{w0}) has the property that
  \bas
	\Big| w_0(s)-\frac{\mu s}{n} \Big|
	&=& \bigg| \int_0^{s^\frac{1}{n}} r^{n-1} u_0(r) dr - \frac{\mu s}{n} \bigg| \\
	&=& \bigg| \int_0^{s^\frac{1}{n}} r^{n-1} \cdot \Big( u_0(r)-\mu\Big) dr \bigg| \\
	&\le& \eps \int_0^{s^\frac{1}{n}} r^{n-1} dr \\
	&=& \frac{\eps}{n} \cdot s
	\qquad \mbox{for all } s\in (0,R^n),
  \eas
  so that since $\Theta(s)\ge \frac{c_1 R^n}{2}$ for all $s\in (0,\frac{R^n}{2}]$ by (\ref{3.3}), we obtain that
  \be{3.5}
	\Big| w_0(s)-\frac{\mu s}{n} \Big|
	\le \frac{\eps}{n} \cdot \frac{2}{c_1 R^n} \cdot s\Theta(s)
	= \eps_0 s\Theta(s)
	\qquad \mbox{for all } s\in \Big(0,\frac{R^n}{2}\Big]
  \ee
  according to (\ref{3.4}).
  For larger $s$, we may use that $w_0(R^n)=\frac{\mu R^n}{n}$ to similarly deive from (\ref{3.3}) and (\ref{3.4}) that
  \bas
	\Big| w_0(s)-\frac{\mu s}{n} \Big|
	&=& \bigg| \bigg\{ w_0(R^n)-\int_{s^\frac{1}{n}}^R r^{n-1} u_0(r) dr \bigg\} - \frac{\mu s}{n}\bigg| \\
	&=& \bigg| \frac{\mu \cdot (R^n-s)}{n} - \int_{s^\frac{1}{n}}^R r^{n-1} \cdot \Big(u_0(r)-\mu+\mu\Big)dr \bigg| \\
	&=& \bigg| - \int_{s^\frac{1}{n}}^R r^{n-1} \cdot \Big(u_0(r)-\mu\Big) dr \bigg| \\
	&\le& \eps \int_{s^\frac{1}{n}}^R r^{n-1} dr \\
	&=& \frac{\eps\cdot (R^n-s)}{n} \\
	&\le& \frac{\eps}{nc_1} \cdot \Theta(s) \\
	&=& \eps_0 \cdot \frac{R^n}{2} \Theta(s) \\
	&\le& \eps_0 s\Theta(s)
	\qquad \mbox{for all } s\in \Big(\frac{R^n}{2},R^n\Big).
  \eas
  Together with (\ref{3.5}), this shows that (\ref{2.2}) holds, so that Lemma \ref{lem2} becomes applicable so as to warrant that due
  to our definition of $\alpha$, the function $w$ from (\ref{w}) satisfies
  \be{3.6}
	\Big| w(s,t)-\frac{\mu s}{n}\Big| \le \eps_0 s \Theta(s) e^{-\alpha t}
	\qquad \mbox{for all $s\in (0,R^n)$ and each } t\in (0,\tm).
  \ee
  As a first consequence, this particularly ensures that
  \bas
	\frac{w(s,t)}{s} \le \frac{\mu}{n} + \eps \|\Theta\|_{L^\infty((0,R^n))}
	\qquad \mbox{for all $s\in (0,R^n)$ and } t\in (0,\tm),
  \eas
  whence the outcome of Lemma \ref{lem79} when combined with Lemma \ref{lem_loc} asserts that $\tm=\infty$, and that there exists
  $c_2=c_2(u_0)>0$ satisfying
  \be{3.66}
	\|u(\cdot,t)\|_{L^\infty(B_R)} \le c_2
	\qquad \mbox{for all } t>0.
  \ee
  Through standard elliptic regularity theory, applied to the second equatio in (\ref{0}), for each $p>1$ this implies the existence
  of $c_3(p)=c_3(p;u_0)>0$ such that
  \be{3.7}
	\|v(\cdot,t)\|_{W^{2,p}(B_R)} \le c_3(p)
	\qquad \mbox{for all } t>0,
  \ee
  whereupon recalling maximal Sobolev regularity theory for the inhomogeneous linear heat equation
  $u_t=\Delta u - \nabla v\cdot\nabla u - \Delta v\cdot u$ (\cite{giga_sohr}) shows that in view of a known property of
  embedding into $C^{1+\beta,\beta}$ spaces (\cite{amann_emb}), employing (\ref{3.7}) for suitably large $p$ yields
  $c_4=c_4(u_0)>0$ fulfilling
  \bas
	\|\nabla u(\cdot,t)\|_{L^\infty(B_R)} \le c_4
	\qquad \mbox{for all } t>1.
  \eas
  As, by definition of $w$ and $z$,
  \bas
	z_{ss}(s,t)=w_{ss}(s,t)=\frac{1}{n^2} s^{\frac{1}{n}-1} u_r( s^\frac{1}{n},t)
	\qquad \mbox{for all $s\in (0,R^n)$ and } t>0,
  \eas
  the latter means that
  \bea{3.8}
	\|z_{ss}(\cdot,t)\|_{L^q((0,R^n))}
	&=& \frac{1}{n^{2q}} \int_0^{R^n} s^{(\frac{1}{n}-1) q} \big| u_r(s^\frac{1}{n},t)\big|^q ds \nn\\[2mm]
	&\le& c_5
	\qquad \mbox{for all } t>1,
  \eea
  with $c_5\equiv c_5(u_0):=\frac{c_4}{n^{2q}} \int_0^{R^n} s^{(\frac{1}{n}-1)q} ds$ being finite since $q<\frac{n}{n-1}$.\abs
  We now once more go back to (\ref{3.6}), which when restated in terms of the variable $z$ states that
  \be{3.9}
	\|z(\cdot,t)\|_{L^\infty((0,R^n))} \le c_6 e^{-\alpha_0 t}
	\qquad \mbox{for all } t>0
  \ee
  with $c_6\equiv c_6(m,R):=\eps_0 R^n \|\Theta\|_{L^\infty((0,R^n))}$. Thus, if in accordance with a Gagliardo-Nirenberg inequality
  we take $c_7=c_7(R)>0$ in such a way that
  \bas
	\|\varphi_s\|_{L^\infty((0,R^n))} \le c_7\|\varphi_{ss}\|_{L^q((0,R^n))}^\gamma \|\varphi\|_{L^\infty((0,R^n))}^{1-\gamma}
	\quad \mbox{for all $\varphi\in W^{2,q}((0,R^n))$ fulfilling } \varphi(0)=\varphi(R^n)=0,
  \eas
  then a combination of (\ref{3.8}) with (\ref{3.9}) reveals that
  \be{3.10}
	\|z_s(\cdot,t)\|_{L^\infty((0,R^n))}
	\le c_7\cdot c_5^\gamma \cdot c_6^{1-\gamma} e^{\alpha_0 (1-\gamma) t}
	\qquad \mbox{for all } t>1.
  \ee
  It therefore only remains to note that $nz_s(s,t)=u(s^\frac{1}{n},t)-\mu$ for all $s\in (0,R^n)$ and $t>0$,
  to conclude (\ref{3.2}) from (\ref{3.10}) and (\ref{3.66}) for some suitably large $C=C(u_0)>0$ and $\alpha$ as determined in
  (\ref{3.4}).
\qed
We have thereby accomplished the first of our objectives:\abs
\proofc of Theorem \ref{theo5} i). \quad
  This statement immediately results from Lemma \ref{lem3}.
\qed
\mysection{Blow-up detection via refined moment analysis}\label{sect4}
Once more on the basis of (\ref{0z}),
and again relying on essential properties of the function $\Theta$ from Lemma \ref{lem1},
we can finally make sure that under an assumption on the mass level which is essentially complementary to that
from Theorem \ref{theo5} i),
in the presence of any nontrivial nonnegative initial data in (\ref{0z}),
and hence for arbitrarily small concentration-increasing perturbations of homogeneity,
a suitably designed moment-like functional cannot exist globally:
\begin{lem}\label{lem4}
  Let $n\ge 2$ and $R>0$, and assume that $u_0$ satisfies (\ref{init}) and is such that with $m:=\int_{B_R} u_0$,
  $\mu:=\frac{nm}{\omega_n R^n}$ and $w_0$ and $\xis$ as in (\ref{w0}) and Lemma \ref{lem1} we have (\ref{5.1})
  as well as
  \be{4.2}
	w_0(s) \ge \frac{\mu s}{n}
	\quad \mbox{for all } s\in (0,R^n)
	\qquad \mbox{but} \qquad
	w_0-\frac{\mu (\cdot)}{n} \not\equiv 0.
  \ee
  Then $\tm<\infty$.
\end{lem}
\proof
  Assuming on the contrary that $(u,v)$ be globally defined, from the regularity properties asserted by Lemma \ref{lem_loc}
  we readily obtain that then, with $\Theta$ taken from Lemma \ref{lem1},
  \bas
	\Phi(t):=\int_0^{R^n} s^{\frac{2}{n}-1} \Theta(s) z(s,t) ds,
	\qquad t\ge 0,
  \eas
  defines a function $\Phi \in C^0([0,\infty)) \cap C^1((0,\infty))$ which due to (\ref{0z}) satisfies
  \bea{4.3}
	\Phi'(t)
	&=& n^2 \int_0^{R^n} s\Theta(s) z_{ss}(s,t) ds
	+ n\int_0^{R^n} s^{\frac{2}{n}-1} \Theta(s) z(s,t)z_s(s,t) ds \nn\\
	& & + \mu \int_0^{R^n} s^{\frac{2}{n}-1} \Theta(s) z(s,t) ds
	\qquad \mbox{for all } t>0.
  \eea
  Here since $z(0,t)=z(R^n,t)=0$ for all $t>0$, and since $\Theta(R^n)=\lim_{s\searrow 0} s\Theta(s)=\lim_{s\searrow 0} s\Theta_s(s)=0$
  according to Lemma \ref{lem1}, two integrations by parts show that
  \bas
	& & \hs{-20mm}
	n^2 \int_0^{R^n} s\Theta(s) z_{ss}(s,t) ds \\
	&=& - n^2 \int_0^{R^n} \Big\{ s\Theta_s(s)+\Theta(s)\Big\} \cdot z_s(s,t) ds
	+ n^2 \cdot \Big[ s\Theta(s)z_s(s,t) \Big]_{s=0}^{s=R^n} \\
	&=& n^2 \int_0^{R^n} \Big\{ s\Theta_{ss}(s)+2\Theta_s(s)\Big\} \cdot z(s,t) ds
	-n^2\cdot \Big[ \Big\{ s\Theta_s(s)+\Theta(s)\Big\} z(s,t) \Big]_{s=0}^{s=R^n} \\
	&=& n^2 \int_0^{R^n} \Big\{ - \frac{\xis^2}{n^2R^2} s^{\frac{2}{n}-1} \Theta(s)\Big\} \cdot z(s,t) ds \\
	&=& - \frac{\xis^2}{R^2} \Phi(t)
	\qquad \mbox{for all } t>0
  \eas
  thanks to (\ref{1.2}).
  Moreover, using the identities $\Theta(R^n)=0$ and $\lim_{s\searrow 0} \big\{ s^{\frac{2}{n}-1} \Theta(s) z^2(s,t)\big\}=0$
  for all $t>0$, the latter being ensured by boundedness of $\Theta$ and of the fact that
  $z_s(\cdot,t)=\frac{1}{n} u\big( (\cdot)^\frac{1}{n},t\big)-\frac{\mu}{n}$ for each fixed $t>0$,
  we may rely on the inequalities $\Theta\ge 0$ and $\Theta_s\le 0$ asserted by Lemma \ref{lem1} to see upon another integration
  by parts that
  \bas
	n \int_0^{R^n} s^{\frac{2}{n}-1} \Theta(s) z(s,t) z_s(s,t) ds
	&=& \frac{n}{2} \int_0^{R^n} s^{\frac{2}{n}-1}\Theta(s) (z^2)_s(s,t) ds \\
	&=& \frac{n-2}{2} \int_0^{R^n} s^{\frac{2}{n}-2} \Theta(s) z^2(s,t) ds
	- \frac{n}{2} \int_0^{R^n} s^{\frac{2}{n}-1} \Theta_s(s) z^2(s,t) ds \nn\\
	& & + \frac{n}{2} \cdot \Big[ s^{\frac{2}{n}-1} \Theta(s) z^2(s,t)\Big]_{s=0}^{s=R^n} \\[2mm]
	&\ge& 0
	\qquad \mbox{for all } t>0,
  \eas
  because $n\ge 2$.
  Therefore, (\ref{4.3}) implies that
  \bas
	\Phi'(t)
	\ge \Big(\mu-\frac{\xis^2}{R^2}\Big) \Phi(t)
	\qquad \mbox{for all } t>0,
  \eas
  so that since (\ref{5.1}) guarantess that $\mu>\frac{\xis^2}{R^2}$, and since (\ref{4.2}) along with the continuity of the nontrivial
  nonnegative function $z(\cdot,0)=w_0-\frac{\mu (\cdot)}{n}$ warrants positivity of $\Phi(0)$, we conclude that
  \bas
	\Phi(t) \ge \Phi(0) e^{(\mu-\frac{\xis^2}{R^2})t} \to +\infty
	\qquad \mbox{as } t\to\infty.
  \eas
  However, the observation that $z(s,t) \le w(s,t) \le \frac{\mu R^n}{n}$ for all $(s,t)\in (0,R^n)\times (0,\infty)$, and that thus
  \bas	
	\Phi(t) \le \frac{\mu R^n}{n} \|\Theta\|_{L^\infty((0,R^n))} \int_0^{R^n} s^{\frac{2}{n}-1} ds
	= \frac{\mu R^{n+2}}{2} \|\Theta\|_{L^\infty((0,R^n))}
	\qquad \mbox{for all } t>0,
  \eas
  shows that this is impossible and that hence, in fact, $\tm$ must have been finite.
\qed
Our argument thereby becomes complete:\abs
\proofc of Theorem \ref{theo5} ii). \quad
  Observing that (\ref{5.2}) precisely asserts that the accordingly defined function $w_0$ from (\ref{w0}) satisfies (\ref{4.2}),
  we readily obtain the claim as a consequence of Lemma \ref{lem4} when combined with Lemma \ref{lem_loc}.
\qed

\bigskip

{\bf Acknowledgement.} \quad
The first author was supported by the
National Natural Science Foundation of China (No.~12371210) and Natural Science Foundation of Sichuan Province (No.~26QNJJA0020). The second author acknowledges support of the {\em Deutsche Forschungsgemeinschaft}
(No.~462888149).


\begin{thebibliography}{99}
%
\bibitem{amann_emb}
  \sc Amann, H.: \it Compact embeddings of vector-valued Sobolev and Besov spaces.
  \rm Glasnik Mat. {\bf 35} (55), 161-177 (2000)
\bibitem{biler1998}
  \sc Biler, P.: \it Local and global solvability of some parabolic systems modelling chemotaxis.
  \rm Adv.~Math.~Sci.~Appl. {\bf 8}, 715-743 (1998)
\bibitem{biler_cieslak_karch_zienkiewicz_DCDS2017}
  \sc Biler, P., Cie\'slak, T., Karch, G., Zienkiewicz, J.:
  \it Local criteria for blowup in two-dimensional chemotaxis models.
  \rm Discr.~Contin.~Dyn.~Syst. {\bf 37}, 1841-1856 (2017)
\bibitem{blanchet_dolbeault_perthame_EJDE2006}
  \sc Blanchet, A., Dolbeault, J., Perthame, B.:
  \it Two-dimensional Keller-Segel model: Optimal critical mass and qualitative properties of the solutions.
  \rm Electron.~J.~Differential Eq. {\bf 44}, 32 pp. (2006)
\bibitem{yuxiang}
  \sc Dai, C., Li, Y., Yan, J.:
  \it Critical mass phenomenon in a nonlinear Keller-Segel system with indirect signal production.
  \rm J.~Evol.~Equ. {\bf 26}, 61 (2026)
\bibitem{ding_win}
  \sc Ding, M., Winkler, M.:
  \it Radial blow-up in quasilinear Keller-Segel systems: approaching the full picture.  
  \rm Nonlinearity {\bf 37}, 125006 (2024)
\bibitem{djie_win}
  \sc Djie, K., Winkler, M.: \it Boundedness and finite-time collapse in a chemotaxis system with volume-filling effect.
  \rm Nonlinear Anal. {\bf 72}, 1044-1064 (2010)
\bibitem{flw}
  \sc Fuhrmann, J., Lankeit, J., Winkler, M.:
  \it A double critical mass phenomenon in a no-flux-Dirichlet Keller-Segel system.
  \rm J.~Math.~Pures Appl. {\bf 162}, 124-151 (2022)
\bibitem{giga_sohr}
  \sc Giga, Y., Sohr, H.: \it Abstract $L^p$ Estimates for the Cauchy Problem with Applications to the
  Navier-Stokes Equations in Exterior Domains.
  \rm J.~Funct.~Anal. {\bf 102}, 72-94 (1991)
\bibitem{herrero_velazquez}
  \sc Herrero, M.~A., Vel\'azquez, J.~J.~L.:
  \it A blow-up mechanism for a chemotaxis model.
  \rm Ann.~Scu.~Norm.~Sup.~Pisa Cl.~Sci. {\bf 24}, 633-683 (1997)
\bibitem{horstmann_wang}
  \sc Horstmann, D., Wang, G.: \it Blow-up in a chemotaxis model without symmetry assumptions.
  \rm Eur.~J.~Appl.~Math. {\bf 12}, 159-177 (2001)
\bibitem{JL}
  \sc J\"ager, W., Luckhaus, S.:
  \it On explosions of solutions to a system of partial differential equations modelling chemotaxis.
  \rm Trans.~Amer.~Math.~Soc. {\bf 329}, 819-824 (1992)
\bibitem{KS}
  \sc Keller, E.F., Segel, L.A.: \it  Initiation of slime mold aggregation viewed as an instability.
  \rm J.~Theor.~Biol. {\bf 26} 399-415 (1970)
\bibitem{kozono_sugiyama_JEE2008}
  \sc Kozono, H., Sugiyama, Y.:
  \it Local existence and finite time blow-up of solutions in the 2-D Keller-Segel system.
  \rm J.~Evol.~Equ. {\bf 8}, 353-378 (2008)
\bibitem{LSU}
  \sc Lady\v{z}enskaja, O. A., Solonnikov, V. A., Ural'ceva, N. N.:
  \it Linear and Quasi-Linear Equations of Parabolic Type.
  \rm American Mathematical Society, 1968
\bibitem{mao_yuxiang}
  \sc Mao, X., Li, Y.:
  \it Critical mass for Keller-Segel systems with supercritical nonlinear sensitivity.
  \rm Math. Models Methods Appl. Sci. {\bf 33}, 2395-2423 (2023)
\bibitem{mao_liu_yuxiang}
  \sc Mao, X., Liu, M., Li, Y.:
  \rm Critical mass for finite-time chemotactic collapse in the critical dimension via comparison. 
  \rm Nonlinearity {\bf 38}, 095026 (2025)
\bibitem{nagai1995}
  \sc Nagai, T.: \it Blow-up of radially symmetric solutions to a chemotaxis system.
  \rm Adv.~Math.~Sci.~Appl. {\bf 5}, 581-601 (1995)
\bibitem{nagai2001}
  \sc Nagai, T.: \it Blowup of Nonradial Solutions to Parabolic-Elliptic Systems Modeling
  Chemotaxis in Two-Dimensional Domains.
  \rm J.~Inequal.~Appl. {\bf 6}, 37-55 (2001)
\bibitem{NSY}
  \sc Nagai, T., Senba, T., Yoshida, K.:
  \it Application of the Trudinger-Moser inequality to a parabolic system of chemotaxis.
  \rm Funkc.~Ekvacioj, Ser.~Int. {\bf 40}, 411-433 (1997)
\bibitem{perthame_book}
  \sc Perthame, B.: \it Transport Equations in Biology.
  \rm Birkh\"auser, Basel, 2007
\bibitem{senba_suzuki}
  \sc Senba, T., Suzuki, T.: \it  Chemotactic collapse in a parabolic-elliptic system of mathematical biology.
  \rm Adv.~Differential Eq. {\bf 6}, 21-50 (2001)
\bibitem{senba_suzuki_JFA}
  \sc Senba, T., Suzuki, T.:
  \it Weak solutions to a parabolic-elliptic system of chemotaxis.
  \rm J.~Funct.~Anal. {\bf 191}, 17-51 (2002)
\bibitem{suzuki_book}
  \sc Suzuki, T.: \it Free Energy and Self-Interacting Particles.
  \rm Birkh\"auser, Boston, 2005
\bibitem{suzuki_JMPA}
  \sc Suzuki, T.: \it
  Exclusion of boundary blowup for 2D chemotaxis system provided with Dirichlet boundary condition for the Poisson part.
  \rm J. Math. Pures Appl. {\bf 100}, 347-367 (2013)
\bibitem{taowin_JEMS}
  \sc Tao, Y., Winkler, M.:
  \it Critical mass for infinite-time aggregation in a chemotaxis model with indirect signal production.
  \rm J.~Eur.~Math.~Soc. {\bf 19}, 3641-3678 (2017)
\bibitem{win_MATANN}
  \sc Winkler, M: \it How unstable is spatial homogeneity in Keller-Segel systems?
  A new critical mass phenomenon in two- and higher-dimensional parabolic-elliptic cases.
  \rm Math.~Ann. {\bf 373}, 1237-1282 (2019)
\bibitem{win_ct_growup_algebraic}
  \sc Winkler, M.: \it Complete infinite-time mass aggregation in a quasilinear Keller-Segel system.
  \rm Israel J.~Math. {\bf 263}, 93-127 (2024)
\bibitem{xiang_yang}
  \sc Xiang, Z., Yang, L.:
  \it Critical mass for the Cauchy problem of a chemotaxis model with indirect signal production mechanism.
 \rm J.~Evol.~Equ.~{\bf 25}, 26 (2025)
\bibitem{yang_lan}
  \sc Yang, L.: \it Critical mass for a no-flux-Dirichlet chemotaxis model with indirect signal production mechanism.
  \rm Adv. Nonlinear Anal. {\bf 15}, 20250147 (2026)
%
\end{thebibliography}
\end{document}